\documentclass[12pt]{amsart}
\usepackage{amscd,amsfonts,amssymb,amsmath}
\usepackage{hyperref}
\usepackage{epsfig}

\newtheorem{theorem}{Theorem}[section]
\newtheorem{corollary}[theorem]{Corollary}
\newtheorem{lemma}[theorem]{Lemma}
\newtheorem{proposition}[theorem]{Proposition}
\newtheorem{example}[theorem]{Example}

\theoremstyle{remark}
\newtheorem{remark}[theorem]{Remark}
\newtheorem{definition}[theorem]{Definition}

\newtheorem{question}[theorem]{Question}

\newtheorem*{ack}{Acknowledgement}

\usepackage[all,cmtip]{xy}
\usepackage{xcolor}

\newcommand{\Alex}{\operatorname{Alex}}
\newcommand{\Aut}{\operatorname{Aut}}
\newcommand{\Out}{\operatorname{Out}}

\newcommand{\Conj}{\operatorname{Conj}}
\newcommand{\Core}{\operatorname{Core}}

\newcommand{\Inn}{\operatorname{Inn}}
\newcommand{\Hol}{\operatorname{Hol}}

\newcommand{\T}{\operatorname{T}}

\newcommand{\R}{\operatorname{R}}

\newcommand{\Z}{\operatorname{Z}}

\newcommand{\id}{\mathrm{id}}

\newcommand{\conj}{\triangleright}

\begin{document}
\title{Constructions and automorphisms of quandles}
\author{Valeriy G. Bardakov}
\author{Birama Sangare}

\date{\today}

\address{Sobolev Institute of Mathematics, 4 Acad. Koptyug avenue, 630090, Novosibirsk, Russia.}
\email{bardakov@math.nsc.ru}

\address{Novosibirsk State University, 2 Pirogova Street, 630090, Novosibirsk, Russia.}
\email{b.sangare@g.nsu.ru}


\begin{abstract}
We prove that if $M$ is an abelian group and $X=\Alex(M,\varphi)$ is a Latin Alexander quandle
 with $\id-\varphi\in\Aut(M)$, then  $\Aut(X)\cong M\rtimes C_{\Aut(M)}(\varphi)$,
  $\Inn(X)\cong M\rtimes\langle\varphi\rangle$ and
  $\Out(X)\cong C_{\Aut(M)}(\varphi)/\langle\varphi\rangle$.
  As a consequence, we obtain a description of the automorphism groups of direct
  products of the cyclic commutative medial quandles  $C_{2n+1}$ and the dihedral quandles $R_{2n+1}$, and we find their
  orders. 

Furthermore, we correct a result of Bardakov--Fedoseev [Algebra and Logic \textbf{63} (2024), no.~2, 111--142] on multiplication of quandles
and prove that the Borodin--Neshchadim--Simonov and Markhinina--Nasybullov quandles can be constructed
 from conjugation, core, and generalized Alexander quandles, using the operation of multiplication of quandles.

Finally, we consider the quotients of a quandle $Q$, namely the medialization $Q^{\mathrm{med}}$, the commutativization $Q^{\mathrm{com}}$, and the trivialization $Q^{\mathrm{tr}}$.
 We prove that the conjugation quandle $\Conj(G)$ is medial if and only if   $\gamma_2(G)\le Z(G)$; we find bounds for $|\Conj(G)^{\mathrm{med}}|$ and prove that  $\Core(G)^{\mathrm{com}}\cong\Core(G/\langle\!\langle g^3\rangle\!\rangle)$. We describe commutative generalized Alexander quandles and, as a consequence, find the commutativization of the dihedral quandle: $R_n^{\mathrm{com}}\cong R_{\gcd(n,3)}$.

\end{abstract}

\maketitle
\section{Introduction}

A quandle is an algebraic system with a single binary operation that satisfies axioms encoding the three Reidemeister moves of planar diagrams of links in 3-space. These objects appear in diverse areas of mathematics, namely knot theory \cite{Joyce, Matveev}, group theory, set-theoretic solutions to the Yang--Baxter equations and Yetter--Drinfeld modules \cite{Eisermann}, Riemannian symmetric spaces \cite{Loos}, and Hopf algebras \cite{Andruskiewitsch}, to name a few. Though already studied under different guises in the literature, the study of these objects gained momentum after the foundational works of Joyce \cite{Joyce} (see also \cite{Joyce-Thesis}) and Matveev \cite{Matveev}, who showed that the knot quandle is a complete invariant of tame knots up to a simultaneous reversal of the orientations of the knot and the ambient 3-space.
 Although link quandles are strong invariants, it is difficult to check whether two quandles are isomorphic. This motivated a search for new properties and invariants of quandles themselves. We refer the reader to the articles \cite{Carter, Kamada, Nelson} for more on the historical development of the subject.

 Jedli\v{c}ka, Pilitowska, Stanovsk\'y, and Zamojska-Dzienio \cite{PADA} developed a structural theory of medial quandles using affine meshes, proving that every medial quandle is built from affine pieces via a heterogeneous affine structure. More recently, Szymik observed that the abelianization of the \emph{structure group} of a quandle is the free abelian group on its set of orbits, and computed this for Thompson's quandle \cite{MS}. This is a different invariant from the one introduced in Section~\ref{sec-quotients}.
 

Automorphisms of quandles, which reveal a lot about their internal structures, have been investigated in great detail in a series of papers \cite{BDS, BarTimSin, Elhamdadi2012}. By fusing ideas from combinatorial group theory into quandles, the works \cite{BSS1, BSS2} show that free quandles and link quandles are residually finite.

Commutative quandles form a class of quandles that enjoy rich algebraic properties \cite{D}. In \cite{BE}, a problem was formulated to describe finitely generated commutative quandles. The automorphism groups of commutative quandles reveal a complex interplay between symmetry and geometry. In particular, for medial commutative quandles, the automorphism group provides a direct link to the theory of modules, since these quandles are often isomorphic to affine structures over commutative rings \cite{LT, M}.

Medial quandles are studied in \cite{PADA}. It was shown in \cite{LT} that the theory of finitely generated commutative medial quandles is similar to the theory of finitely generated abelian groups. Also, in this paper an answer was given to the question from \cite{BE}. It was proven that every cyclic commutative medial quandle is isomorphic to an Alexander quandle $\Alex(M,\varphi)$ such that the map $\varphi$ is multiplication by $1/2$, where $M$ is an additive abelian group containing $1/2$.

The automorphism group of a general Alexander quandle was described by Hou \cite{Hou}: when $\id-\varphi$ is an automorphism of $M$, one has $\Aut(\Alex(M,\varphi))\cong M\rtimes C_{\Aut(M)}(\varphi)$, where $C_{\Aut(M)}(\varphi)$ is the centralizer of $\varphi$ in $\Aut(M)$. In this paper we revisit this result, describe explicitly the groups $\Aut$, $\Inn$, and $\operatorname{Out}$, and apply the description to direct products of cyclic commutative medial quandles and of dihedral quandles.

In \cite[Corollary 10.4]{Joyce}, it was proved that the two-generated free involutory quandle is isomorphic to $\Core(\mathbb{Z})$, with generators $0$ and $1$, and in \cite[Theorem 10.5]{Joyce}, a structure of a free medial involutory quandle with $n+1$ generators was found.

This paper is organized in the following way. In Section \ref{sec-prelim}, some basic definitions and examples of the theory of quandles are recalled.

In Section \ref{Sec3}, we explicitly describe the automorphism group, the inner automorphism group, and the outer automorphism group of a Latin Alexander quandle by revisiting the foundational result established by Hou \cite[Theorem 3.2]{Hou}. We prove that if $M$ is an abelian group and
 $X=\Alex(M,\varphi)$ is a Latin Alexander quandle with $\id-\varphi\in\Aut(M)$, then $\Aut(X)\cong M\rtimes C_{\Aut(M)}(\varphi)$,
  $\Inn(X)\cong M\rtimes\langle\varphi\rangle$ and
  $\Out(X)\cong C_{\Aut(M)}(\varphi)/\langle\varphi\rangle$ (Theorem~3.2).
  As a consequence, we obtain a description of the automorphism groups of direct
  products of the quandles $C_{2n+1}$ and $R_{2n+1}$ (Corollary~3.4) and of their
  orders (via the Hillar--Rhea formula). These results generalize some results from  \cite{BDS, Elhamdadi2012}.
We prove that $C_\infty$ and $C_{2n+1}$ are generated by two elements
  (Proposition~3.1).

In Section \ref{sec-constructions}, we investigate the multiplication of quandle structures defined on a set, introduced by Bardakov and Fedoseev \cite{BF}. Through two independent counterexamples, we show that neither the original distributivity condition nor its mirror counterpart is sufficient on its own to guarantee that the composition yields a quandle. We verify their joint sufficiency in Proposition \ref{Proposition 4.1}, while establishing that these combined relations are not strictly necessary to obtain a valid quandle structure.
 In \cite{BNS,MN}, new constructions of quandles from groups were introduced. Using the multiplication of quandle structures, we prove that these quandles can be constructed from the conjugation quandle $\operatorname{Conj}(G)$, the core quandle $\operatorname{Core}(G)$, and the generalized Alexander quandle $\operatorname{Alex}(G,\varphi)$, where $\varphi$ is an automorphism of $G$ (Theorem \ref{Theorem 4.5}).

  The medialization of a quandle is a standard construction in universal algebra \cite{D, JPZ}.
In Section \ref{sec-quotients} we investigate quotients of quandles: medialization, commutativization
  and trivialization.    We prove that $\Conj(G)$ is medial if and only if
  $\gamma_2(G)\le Z(G)$; we find a bound for $|\Conj(G)^{\mathrm{med}}|$  (Theorem~5.4).
The commutativity of $\Core(G)$ was investigated by Ta \cite{LT1} for groups of exponent $3$. Our Theorem~\ref{theorem 5.6}(2) generalizes this characterization to arbitrary groups. We prove that $\Core(G)^{\mathrm{com}}\cong\Core(G/\langle\!\langle g^3\rangle\!\rangle)$.
As a direct consequence, Corollary~\ref{corollary 5.8} establishes the isomorphism $\R_n^{com}\cong \R_{\gcd(n,3)}$ for dihedral quandles.

 We conclude by formulating some questions for future research (see Section \ref{questions}).

\section{Preliminaries}\label{sec-prelim}

The main facts about racks and quandles can be found in \cite{BES}. A \textit{quandle} is a non-empty set $Q$ with a binary operation $(x,y) \mapsto x * y$ satisfying the following axioms:
\begin{enumerate}
\item[(Q1)] $x * x = x$ for all $x \in Q$ (idempotency),
\item[(Q2)] for any $x,y \in Q$ there exists a unique $z \in Q$ such that $x=z*y$ (right-invertibility),
\item[(Q3)] $(x * y) * z = (x * z) * (y * z)$ for all $x, y, z \in Q$ (right self-distributivity).
\end{enumerate}

 For all $x \in Q$, the right-multiplication map $S_x \colon Q \to Q$ defined by $y \mapsto y\ast x$ is an automorphism.
These automorphisms are called inner automorphisms, and the group generated by all such automorphisms is denoted by $\Inn(Q)$.

An algebraic system satisfying axioms $Q2$ and $Q3$ is called a \textit{rack}. Many interesting examples of quandles come from groups. We give some examples.

\begin{itemize}
\item A quandle $Q$ is called \textit{trivial} if $x*y=x$ for all $x, y \in Q$. Unlike groups, a trivial quandle can have an arbitrary number of elements. We denote the trivial quandle on $n$ elements by $\T_n$.
\item If $G$ is a group, then the binary operation $x \conj y = y^{-1} x y$ turns $G$ into the quandle $\Conj(G)$ called the \textit{conjugation quandle} of $G$.
\item A group $G$ with the binary operation $x*y= y x^{-1} y$ turns the set $G$ into the quandle $\Core(G)$ called the \textit{core quandle} of $G$. In particular, if $G= \mathbb{Z}_n$, the cyclic group of order $n$, then it is called the \textit{dihedral quandle} and is denoted by $\R_n$.
\item Let $G$ be a group and $\varphi \in \Aut(G)$. Then the set $G$ with binary operation $x * y = \varphi(xy^{-1})y$ forms a quandle $\Alex(G, \varphi)$ called the \textit{generalized Alexander quandle} of $G$ with respect to $\varphi$. If $G$ is an additive abelian group, the quandle $\Alex(G, \varphi)$ defined by
\[
x\ast y= \varphi(x)+(\id-\varphi)(y),
\]
is called an \textit{Alexander quandle}.
\end{itemize}

A quandle is said to be \textit{connected} if it admits a transitive action by its group of inner automorphisms. For example, dihedral quandles of odd order are connected, whereas those of even order are disconnected.

Let us recall some special classes of quandles.
\begin{enumerate}
\item[(1)] A quandle $(Q,\ast)$ is \textit{commutative} if $x\ast y=y\ast x$ for any $x, y \in Q$. For example, the dihedral quandle $\R_3$ is commutative.
\item[(2)] A quandle $(Q,\ast)$ is called \textit{medial} (equivalently \textit{entropic}, or \textit{abelian in the sense of Kurosh}) if $(x\ast y)\ast(t\ast z)=(x\ast t)\ast(y\ast z)$ for any $x, y, t, z \in Q$.

In the sense of Kurosh \cite{K}, medial quandles are exactly the abelian algebras (entropic algebras) in universal algebra. In this paper we use the term ``medial quandle'' exclusively, to avoid any confusion with the notion of abelianity in the commutator theory. In particular, any Alexander quandle is medial.

For reference, we recall that an algebra $A$ of signature $\Omega$ is abelian in the sense of Kurosh if and only if the following two conditions hold in $A$:

a) for any $\omega \in \Omega_n$, $\omega' \in \Omega_s$, $n, s \geq 1$,
\[
(x_{11} \ldots x_{1n} \omega) \ldots (x_{s1} \ldots x_{sn} \omega) \omega' = (x_{11} \ldots x_{s1} \omega') \ldots (x_{1n} \ldots x_{sn} \omega') \omega,
\]
where $x_{ij}$, $1 \leq i \leq s$, $1 \leq j \leq n$ is a matrix of variables;

b) all $\omega \in \Omega_0$ fix one element in $A$ and this element forms an $\Omega$-subalgebra of $A$.

\item[(3)] Let $C$ be a ring of dyadic rationals $\mathbb{Z}[\frac{1}{2}]$ or a cyclic group $\mathbb{Z}_{2n+1}$ of order $2n+1$. Then $(C; *)$, where $x\ast y=\frac{1}{2}(x+y)$, is a quandle. Moreover, this quandle is commutative and medial. We will call $(C; *)$ a cyclic commutative medial quandle. Note that Luc Ta \cite{LT} called this quandle \textit{midpoint}. If $C = \mathbb{Z}[\frac{1}{2}]$, then we denote this quandle by $C_{\infty}$. If $C = \mathbb{Z}_{2n+1}$, then we denote this quandle by $C_{2n+1}$. One can see that we have $2^{-1}=n+1$ in $\mathbb{Z}_{2n+1}$ and, hence, the operation in $C_{2n+1}$ can be defined by the rule $x\ast y = (n+1)(x+y)$ for $x,y \in \mathbb{Z}_{2n+1}$.
\item[(4)] A quandle $Q$ is called \textit{Latin} if left multiplication by any element of $Q$ yields a bijective map on $Q$, that is, the map $L_x \colon Q \to Q$ defined by $L_x(y) = x *y$ is a bijection for each $x \in Q$. For example, if $M$ is an additive abelian group, then the Alexander quandle $\operatorname{Alex}(M,\varphi)$ is Latin if and only if $\id -\varphi$ is an automorphism of $M$. Note that every commutative quandle is Latin.
\end{enumerate}

Further, we will use the following well-known proposition, which can be found in \cite{TP}, \cite{PADA}, and in \cite[Proposition 4.1]{LT}.

\begin{proposition}[\cite{TP, PADA, LT}]\label{Proposition 2.3}
Let $Q$ be a medial Latin quandle. Then there exist an additive abelian group $M$ and an automorphism $\varphi\in\Aut(M)$ such that $\id-\varphi\in\Aut(M)$ and
\[
Q\cong\Alex(M,\varphi).
\]
\end{proposition}

This is the Toyoda--Bruck theorem for medial Latin quandles. By this proposition, every cyclic commutative medial quandle is isomorphic to an Alexander quandle $\Alex(M, \varphi)$ in which $M$ is an additive abelian group containing $\frac{1}{2}$ and the map $\varphi = \hat{\frac{1}{2}}$ is the multiplication by $\frac{1}{2}$.

\section{Automorphism groups of Latin Alexander quandles }\label{Sec3}

The structural investigation of automorphism groups of affine and Alexander quandles builds upon important developments in algebraic quandle theory. In a foundational result, Hou \cite{Hou} determined the full automorphism group of general affine quandles, characterizing their structure as split extensions over the underlying abelian module. In this section, we build upon Hou's framework to detail the automorphism, inner automorphism, and outer automorphism groups of Latin Alexander quandles. We then apply this algebraic description to determine the automorphism groups and evaluate the exact orders for direct products of cyclic commutative medial quandles $C_{2n+1}$ and dihedral quandles $\R_{2n+1}$.

At first, we prove that $C_{\infty}$ and $C_{2n+1}$ are $2$-generated quandles. We note that the subquandle $\langle 0,1\rangle$ is, by convention, closed under
the inverse operation $*^{-1}$: we define the subquandle generated by a subset as the
smallest subset closed under both $*$ and $*^{-1}$.

\begin{proposition}\label{Proposition 3.2}
Let $C_{\infty}$ and $C_{2n+1}$ be the cyclic commutative medial quandles defined above. Then

1) The map $\pi \colon C_{\infty} \to C_{2n+1}$ defined by $\pi \left(\frac{a}{2^k}\right) = \overline{a(n+1)^k}$ is a well-defined \textit{epimorphism} of quandles.

2) $C_{\infty} = \langle 0, 1 \rangle$ and $C_{2n+1} = \langle \bar{0}, \bar{1} \rangle$.
\end{proposition}

\begin{proof}

1) We are considering the unital commutative and associative rings
\[
\mathbb{Z}\left[\frac{1}{2}\right] = \left\{ \frac{a}{2^k}~|~a \in \mathbb{Z}, k \in \mathbb{N}_0\right\},
\]
and
\[
\mathbb{Z}_{2n+1} = \{ \bar{0}, \bar{1}, \ldots, \overline{2n} \}
\]
with the standard ring operations. Since $2$ is invertible in $\mathbb{Z}_{2n+1}$, there is a
unique unital ring homomorphism $\mathbb{Z}[\tfrac12]\to\mathbb{Z}_{2n+1}$, sending $\tfrac12$ to
$n+1$. This map is the map $\pi$ of the statement.

\emph{Well-definedness of $\pi$.} If $a/2^k=b/2^l$ in $\mathbb{Z}[\tfrac12]$, assume without loss of generality that $k\ge l$; then $a=b\cdot 2^{k-l}$, so
\[
a(n+1)^k = b\cdot 2^{k-l}(n+1)^k = b(n+1)^l\cdot\bigl(2(n+1)\bigr)^{k-l}.
\]
Since $2(n+1) = 2n+2 \equiv 1 \pmod{2n+1}$, it follows that $\bigl(2(n+1)\bigr)^{k-l} \equiv 1 \pmod{2n+1}$. Hence, we obtain $\overline{a(n+1)^k} = \overline{b(n+1)^l}$, proving that $\pi$ is well-defined.

\emph{Additivity, multiplicativity, compatibility with $*$.} These follow directly from the fact that $\pi$ is a unital ring homomorphism and from the definition of the operational law $x*y = \tfrac12(x+y)$ on the quandle components.

So $\pi$ is a quandle homomorphism.

To establish the surjectivity of $\pi$, we observe that the restriction of $\pi$ to the subring of integers $\mathbb{Z} \subset \mathbb{Z}[\tfrac{1}{2}]$ is exactly the canonical quotient map $m \mapsto m \pmod{2n+1}$. Since this canonical projection from $\mathbb{Z}$ onto $\mathbb{Z}_{2n+1}$ is surjective, the standard inclusion implies that $\pi(\mathbb{Z}[\tfrac{1}{2}])$ covers the entire target ring, proving that $\pi$ is surjective.

2) We denote the quandle generated by $0$ and $1$ by $\langle 0, 1\rangle$ with the same operation in $C_{\infty}$. The inverse operation $*^{-1}$ in $C_{\infty}$ is defined by the rule $x *^{-1} y = 2 x - y$. We need to prove $C_{\infty}\subset \langle 0, 1\rangle$.

\emph{Step 1: $\Z\subseteq\langle 0,1\rangle$.} We have $0\ast^{-1}1=-1$, so $-1\in\langle 0, 1\rangle$; $1\ast^{-1}0=2$, so $2\in\langle 0, 1\rangle$; $0\ast^{-1}2=-2$, so $-2\in\langle 0, 1\rangle$; $2\ast^{-1}1=3$, so $3\in\langle 0, 1\rangle$. We argue by strong induction on $|m|$: if $m\in\langle 0, 1\rangle$, then
\[
2m=m\ast^{-1}0\in\langle 0, 1\rangle,\qquad 2m-1=m\ast^{-1}1\in\langle 0, 1\rangle,\qquad -m=0\ast^{-1}m\in\langle 0, 1\rangle.
\]
This covers all integers.

\emph{Step 2: dyadic rationals.} If $a/2^m\in\langle 0,1\rangle$, then
\[
\frac{a}{2^{m+1}}=\frac{a}{2^m}\ast 0\in\langle 0, 1\rangle.
\]
By induction on $m$, all dyadic rationals lie in $\langle 0, 1\rangle$.

Since $0,1\in C_\infty$, we get $C_\infty=\langle 0, 1\rangle$. Finally, by (1),
\[
C_{2n+1} =\pi(C_{\infty})=\pi(\langle 0, 1\rangle)=\langle \pi(0), \pi(1)\rangle =\langle \bar{0}, \bar{1}\rangle.
\]
\end{proof}

Let $\operatorname{Aut}(M)$ be the group of automorphisms of an abelian group $M$. Let us consider the semidirect product $M \rtimes \operatorname{Aut}(M)$, in which multiplication is defined by the rule
\begin{equation} \label{product}
(m_1,\varphi_{1}) \cdot (m_2,\varphi_{2}) = (m_1 + \varphi_1(m_2), \varphi_{1}\varphi_{2})
\end{equation}
for all $m_1, m_2 \in M$ and $\varphi_{1}, \varphi_{2} \in \operatorname{Aut}(M)$. This group $M \rtimes \operatorname{Aut}(M)$ is called the holomorph of $M$ and is denoted by $\Hol(M)$.

Let $X=\operatorname{Alex}(M, \varphi)$ be an Alexander quandle, in which $\id-\varphi$ is an automorphism of $M$. In this case $X$ is a Latin (hence connected) quandle (see \cite{S}). Let $C_{\operatorname{Aut}(M)}(\varphi)$ be the group of automorphisms of $M$ that commute with $\varphi$. Then we consider the semidirect product $\Hol_{\varphi}(M) = M \rtimes C_{\operatorname{Aut}(M)}(\varphi)$, in which multiplication is defined by the rule (\ref{product}). We introduce the set $H$ of all translations by elements of $M$, i.e.,
\[
H=\{t_{m}: M\to M\,\,|\,\,m\in M,\ t_{m}(x)=x+m,\ x\in M \}\cong M.
\]

The multiplication in the semidirect product $H\rtimes \langle \varphi \rangle$ is defined by the rule
\begin{equation} \label{product1}
(t_{m_1},\varphi^k) \cdot (t_{m_2},\varphi^l) = (t_{m_1} + \varphi^k(m_2), \varphi^{k+l})
\end{equation}
for all $m_1, m_2 \in M$ and $k, l \in \mathbb{Z}$, which follows from (\ref{product}).

\begin{theorem}  \label{Theorem 3.2}
Let $X=\operatorname{Alex}(M, \varphi)$ be an Alexander quandle, in which $\id-\varphi$ is an automorphism of $M$. Then

1) $\operatorname{Aut}(X) \cong \Hol_{\varphi}(M)$.

2) Under the isomorphism in 1), the inner automorphism group $\Inn(X)$ corresponds to the specific subgroup:
\[
\Inn(X) = H \rtimes \langle \varphi \rangle.
\]
3) $\operatorname{Out}(X)\cong C_{\Aut(M)}(\varphi)/\langle \varphi \rangle$.
\end{theorem}
\begin{proof}
1) See \cite[Proposition 2.1]{Hou}.

2) We first show that $\langle \varphi \rangle$ normalizes $H$ within $\Aut(X)$. Indeed, for any $t_m \in H$ and $x \in M$, we have:
\[
(\varphi \circ t_m \circ \varphi^{-1})(x) = \varphi(\varphi^{-1}(x) + m) = x + \varphi(m) = t_{\varphi(m)}(x).
\]
Thus, $\varphi t_m \varphi^{-1} = t_{\varphi(m)} \in H$, which implies that $H$ is normalized by $\langle \varphi \rangle$. Consequently, the set $H \rtimes \langle \varphi \rangle$ forms a well-defined subgroup of $\Aut(X)$.

Next, we show that the inner automorphism group $\Inn(X)$ is contained in $H \rtimes \langle \varphi \rangle$. The group $\Inn(X)$ is generated by the right multiplications $S_a$ and their inverses $S_a^{-1}$ for $a \in M$. For any $a \in M$, the right multiplication is given by $S_a(x) = x * a = \varphi(x) + (\id - \varphi)(a)$, which can be factored as:
\[
S_a = t_{(\id - \varphi)(a)} \circ \varphi.
\]
Since $t_{(\id - \varphi)(a)} \in H$ and $\varphi \in \langle \varphi \rangle$, we have $S_a \in H \rtimes \langle \varphi \rangle$. Similarly, the inverse operation is $S_a^{-1}(x) = \varphi^{-1}(x) + (\id - \varphi^{-1})(a)$, which factors as:
\[
S_a^{-1} = t_{(\id - \varphi^{-1})(a)} \circ \varphi^{-1}.
\]
This shows that $S_a^{-1} \in H \rtimes \langle \varphi \rangle$ as well. Since $H \rtimes \langle \varphi \rangle$ is a subgroup of $\Aut(X)$ and contains all the generators of $\Inn(X)$, it follows that $\Inn(X) \le H \rtimes \langle \varphi \rangle$.

Conversely, we show the reverse inclusion. The identity $S_0(x) = \varphi(x)$ implies that $\varphi \in \Inn(X)$, hence $\langle \varphi \rangle \le \Inn(X)$. Furthermore, for any $a \in M$, the composition $S_a \circ \varphi^{-1}$ yields $S_a(\varphi^{-1}(x)) = x + (\id - \varphi)(a) = t_{(\id - \varphi)(a)}(x)$. Thus, $t_{(\id - \varphi)(a)} \in \Inn(X)$. Since $\id - \varphi$ is an automorphism of $M$, the map $\id - \varphi \colon M \to M$ is surjective, meaning that $\{(\id - \varphi)(a) \mid a \in M\} = M$. Therefore, every translation $t_m$ belongs to $\Inn(X)$, so $H \le \Inn(X)$.

It follows that $H \rtimes \langle \varphi \rangle \le \Inn(X)$. Combined with the previous inclusion, we conclude that $\Inn(X) = H \rtimes \langle \varphi \rangle$. Finally, since $H \cap \langle \varphi \rangle = \{\id\}$ (as any map $\varphi^n = t_m$ satisfies $\varphi^n(0) = 0 \implies m = 0$), this product is strictly internal and semidirect.

3) Since $H\cong M$ and $\langle\varphi\rangle$ is contained in the center of $C_{\Aut(M)}(\varphi)$, the quotient $C_{\Aut(M)}(\varphi)/\langle \varphi \rangle$ is a group. We have
\[
\operatorname{Out}(X)\cong \Hol_{\varphi}(M)/(H\rtimes \langle \varphi \rangle)\cong \Hol_{\varphi}(M)/(M\rtimes \langle \varphi \rangle).
\]
Let us define a surjection $\Phi\colon \Hol_{\varphi}(M) \to C_{\Aut(M)}(\varphi)/\langle \varphi \rangle$ by $\Phi(m, \sigma)=\sigma \langle \varphi \rangle$. For $(m_1,\sigma_1),(m_2,\sigma_2)\in \Hol_{\varphi}(M)$,
\[
\Phi((m_1,\sigma_1)(m_2,\sigma_2))= \Phi(m_1+\sigma_1(m_2),\sigma_1\sigma_2)=\sigma_1\sigma_2\langle \varphi \rangle=(\sigma_1\langle \varphi \rangle)(\sigma_2\langle \varphi \rangle),
\]
so $\Phi$ is a homomorphism. Its kernel is $\operatorname{Ker}\Phi=\{(m,\sigma)\,\,|\,\,\sigma\in\langle \varphi \rangle \}=M\rtimes\langle \varphi \rangle$. By the first isomorphism theorem,
\[
\Hol_{\varphi}(M)/(M\rtimes\langle \varphi \rangle)\cong C_{\Aut(M)}(\varphi)/\langle \varphi \rangle,
\]
i.e., $\operatorname{Out}(X)\cong C_{\Aut(M)}(\varphi)/\langle \varphi \rangle$.
This completes the proof.
\end{proof}

\begin{remark}\label{Remark 3.2}
It was proved in \cite[Theorem $2$.$1$]{Elhamdadi2012} that
\[
\operatorname{Aut}(\R_{n}) \cong \mathbb{Z}_{n}\rtimes\operatorname{Aut}(\mathbb{Z}_{n}) \cong \mathbb{Z}_{n}\rtimes\mathbb{Z}_{n}^{\times},
\]
where $\mathbb{Z}_{n}^{\times}$ is the group of units of the ring $\mathbb{Z}_{n}$, which acts on $\mathbb{Z}_{n}$ by multiplication. Any automorphism $f \in \operatorname{Aut}(\R_{n})$ acts on $\R_n$ by the rule
\[
f(x) = ax + b,\qquad a \in \mathbb{Z}_{n}^{\times}, \quad b \in \mathbb{Z}_{n}.
\]

We claim that this result can be extended to $C_{2n+1}$. Indeed, if we consider $\varphi$ as multiplication by $\frac{1}{2}$, then $\operatorname{Aut}(C_{2n+1}) \cong \mathbb{Z}_{2n+1} \rtimes \mathbb{Z}_{2n+1}^{\times}$ by Theorem \ref{Theorem 3.2}. Therefore, $\operatorname{Aut}(\R_{2n+1})$ is isomorphic to $\operatorname{Aut}(C_{2n+1})$.
\end{remark}

Now we write the automorphism group of a finite commutative medial quandle $Q$. By \cite[Theorem 6.4]{LT},
\[
Q\cong\prod_{i=1}^{m}C_{2n_i+1}.
\]
Let us consider the abelian group $M = \prod_{i=1}^{m}\mathbb{Z}_{2n_i+1}$. Since $2$ is coprime to each factor, the component-wise multiplication by $2^{-1} = n_i+1$ defines an automorphism $\hat{\frac{1}{2}}\colon M \to M$ given by
\[
\hat{\frac{1}{2}}(x_1,\cdots,x_m)=((n_1+1)x_1,\cdots,(n_m+1)x_m).
\]

Consequently, the direct product $Q$ identifies naturally with the Alexander quandle $\operatorname{Alex}(M; \hat{\frac{1}{2}})$ over the direct product of the groups. Thus, the following corollary holds.

\begin{corollary} \label{Corollary 3.4}
Let $Q=\prod_{i=1}^{m}C_{2n_i+1}$ be a finite commutative medial quandle and let $M=\prod_{i=1}^m\Z_{2n_i+1}$ and let $L = \operatorname{lcm}(2n_i+1)$ be the least common multiple of the cyclic group orders. Then
\begin{enumerate}
\item[(1)] $\operatorname{Aut}(Q)\cong M \rtimes \Aut(M)$.
\item[(2)] $\Inn(Q)\cong M \rtimes\langle \hat{\frac{1}{2}}\rangle$, where the cyclic subgroup $\langle \hat{\frac{1}{2}} \rangle$ has order equal to the multiplicative order of $2$ modulo $L$.
\item[(3)] $\operatorname{Out}(Q)\cong \Aut(M)/\langle \hat{\frac{1}{2}} \rangle$.
\end{enumerate}
\end{corollary}

\begin{proof}
By Theorem \ref{Theorem 3.2}, to determine the structures of $\Aut(Q)$, $\Inn(Q)$, and $\operatorname{Out}(Q)$, it suffices to compute the centralizer of $\hat{\tfrac12}$ in $\Aut(M)$.

 Since all components have odd order, $2$ is coprime to $L$, and the global mapping $\hat{\tfrac12}$ acts on the abelian group $M$ simply via scalar multiplication by the single integer $N = (L+1)/2$. Because any group endomorphism of $M$ naturally commutes with scalar multiplication by an integer, every element of $\Aut(M)$ commutes with $\hat{\tfrac12}$. This immediately yields $C_{\Aut(M)}(\hat{\tfrac12})=\Aut(M)$. Statements (1), (2), and (3) then follow directly by substituting this centralizer relation into the corresponding parts of Theorem \ref{Theorem 3.2}.
\end{proof}

\begin{question}
What can be said about $\operatorname{Aut}(\R_n\times \R_m)$ for even $n$ or $m$? For odd $n,m$ the answer is obtained as a direct consequence of Corollary \ref{Corollary 3.4}.
\end{question}

It is well known that the dihedral quandle $\R_n$ is isomorphic to the Alexander quandle $\operatorname{Alex}(\mathbb{Z}_n,-\id)$. When considering a product of even order components, the non-invertibility of $2$ in $\mathbb{Z}_{2n}$ prevents the direct application of the Latin framework to even moduli, since $\id-(-\id) = 2\id$ fails to be an automorphism on $\mathbb{Z}_{2n}$.

For odd moduli, the direct products $C_m \times C_n$ and $\R_m \times \R_n$ are Alexander quandles over the finite abelian group $M = \mathbb{Z}_m \times \mathbb{Z}_n$. Since the defining automorphisms $\varphi_C = \hat{\tfrac12}$ and $\varphi_R = -\id$ act as central scalar multiplications on $M$, their centralizers trivially encompass the full group $\Aut(M)$. Thus, Corollary \ref{Corollary 3.4} immediately yields the structural split extension:
\[
\operatorname{Aut}(C_{m} \times C_{n}) \cong \operatorname{Aut}(\R_{m} \times \R_{n}) \cong (\mathbb{Z}_m\times\mathbb{Z}_n) \rtimes \Aut(\mathbb{Z}_m\times\mathbb{Z}_n).
\]
To provide an explicit computational description, let $M$ be a finite abelian group written as a product of prime-power cyclic groups. Following Hillar and Rhea \cite{HillarRhea}, we decompose $M \cong \prod_{p} M_p$, where each $p$-primary component is ordered by non-decreasing exponents:
\[
M_p \cong \mathbb{Z}_{p^{e_1}} \times \mathbb{Z}_{p^{e_2}} \times \cdots \times \mathbb{Z}_{p^{e_r}}, \quad \text{with } 1 \le e_1 \le e_2 \le \cdots \le e_r.
\]
For each exponent index $k \in \{1, \dots, r\}$, we define the structural index bounds $d_k = \max\{l \mid e_l = e_k\}$ and $c_k = \min\{l \mid e_l = e_k\}$. Based on this standard framework, the exact order of the automorphism group of $M_p$ is given by:
\[
|\Aut(M_p)| = \prod_{k=1}^r \bigl(p^{d_k} - p^{k-1}\bigr) \prod_{j=1}^r p^{e_j(r-d_j)} \prod_{i=1}^r p^{(e_i-1)(r-c_i+1)}.
\]
The total order $|\Aut(M)|$ is then obtained as the product of the orders $|\Aut(M_p)|$ over all prime factors of $|M|$.

\begin{example}\label{Example:AutR3R3}
We evaluate the exact order of the automorphism group $\Aut(\R_3 \times \R_3)$ using the Hillar--Rhea primary component framework. The product quandle $\R_3 \times \R_3$ is an Alexander quandle over the finite abelian group $M = \mathbb{Z}_3 \times \mathbb{Z}_3$. The total order is given by the split extension relation:
\[
|\Aut(\R_3 \times \R_3)| = |M| \cdot |\Aut(M)| = 9 \cdot |\Aut(\mathbb{Z}_3 \times \mathbb{Z}_3)|.
\]
The group $M = \mathbb{Z}_3 \times \mathbb{Z}_3$ consists of a single $p$-primary component for the prime $p=3$ with $r = 2$ and ordered exponents $e_1 = 1, e_2 = 1$. The structural index bounds are evaluated as follows:
\begin{itemize}
\item For $k=1$: $e_1 = 1 \implies d_1 = \max\{l \mid e_l = 1\} = 2$ and $c_1 = \min\{l \mid e_l = 1\} = 1$.
\item For $k=2$: $e_2 = 1 \implies d_2 = \max\{l \mid e_l = 1\} = 2$ and $c_2 = \min\{l \mid e_l = 1\} = 1$.
\end{itemize}
We now compute the three product terms of the formula:
\begin{enumerate}
\item[\rm(i)] The first structural product term yields:
\[
\prod_{k=1}^2 \bigl(3^{d_k} - 3^{k-1}\bigr) = \bigl(3^{d_1} - 3^0\bigr)\bigl(3^{d_2} - 3^1\bigr) = (3^2 - 1)(3^2 - 3) = 8 \cdot 6 = 48.
\]
\item[\rm(ii)] The second power term counts the remaining row factors. Since $d_1 = d_2 = 2$ and $r = 2$, the exponent $r - d_j = 2 - 2 = 0$ for all $j$, which yields:
\[
\prod_{j=1}^2 3^{e_j(2-2)} = 3^0 \cdot 3^0 = 1.
\]
\item[\rm(iii)] The third power term counts the reductions due to higher valency components. Since $e_1 = e_2 = 1$, the term $e_i - 1 = 0$ for all $i$, which yields:
\[
\prod_{i=1}^2 3^{(1-1)(2-c_i+1)} = 3^0 \cdot 3^0 = 1.
\]
\end{enumerate}
Multiplying these three terms together, we recover the classical order of the general linear group over the finite field $\mathbb{F}_3$:
\[
|\Aut(\mathbb{Z}_3 \times \mathbb{Z}_3)| = 48 \cdot 1 \cdot 1 = 48.
\]
Finally, substituting this value back into the split extension relation yields the exact total order:
\[
|\Aut(\R_3 \times \R_3)| = 9 \cdot 48 = 432.
\]
This numerical evaluation confirms the exact value of 432 automorphisms verified by full permutation counting, validating the accuracy of the framework.
\end{example}
\begin{remark}\label{Remark 3.6}
By the Chinese Remainder Theorem, if $m$ and $n$ are relatively prime natural numbers, the canonical ring map $f \colon \mathbb{Z}_{mn} \to \mathbb{Z}_m \times \mathbb{Z}_n$ defined by $f(x \bmod mn) = (x \bmod m, x \bmod n)$ is an isomorphism. This mapping directly induces a quandle isomorphism on the corresponding dihedral quandles, yielding $\R_{mn} \cong \R_{m} \times \R_{n}$ for any coprime integers $m$ and $n$. Furthermore, if $m$ and $n$ are additionally assumed to be odd, the same mapping induces a structural isomorphism on the cyclic commutative medial quandles, namely $C_{mn} \cong C_{m} \times C_{n}$.

Note also that, although $\Aut(\R_{2n+1})\cong\Aut(C_{2n+1})$, the quandles $\R_{2n+1}$ and $C_{2n+1}$ are not isomorphic for $2n+1>3$: the quandle $C_m$ is commutative, while $\R_m$ is not commutative for $m>3$.
\end{remark}

\section{Composition of quandle structures}\label{sec-constructions}

The multiplication of two quandle structures defined on the same set can yield a new quandle operation. Following the construction introduced by Bardakov and Fedoseev \cite{BF}, given two quandles $(Q; \star)$ and $(Q; \odot)$, their composition $\star\odot$ is defined by
\[
a\star\odot\,b=(a\star b)\odot b,\qquad a,b\in Q.
\]
In \cite[Theorem 3.2]{BF}, it was stated that if the operation $\odot$ is distributive with respect to the operation $\star$, namely
\begin{equation}\label{eq:distrib-BF}
(a\star b)\odot c=(a\odot c)\star(b\odot c),\qquad a,b,c\in Q,
\end{equation}
then the composition $\star\odot$ satisfies the self-distributivity axiom, defining a quandle structure $(Q; \star\odot)$.

 A structural analysis of the argument presented in \cite{BF} indicates that the proof overlooks the mirror distributivity relation:
\begin{equation}\label{eq:distrib-mirror}
(u\odot b)\star c=(u\star c)\odot(b\star c),\qquad u,b,c\in Q.
\end{equation}
As shown by the explicit counterexamples in the following subsection, condition \eqref{eq:distrib-BF} alone does not suffice to guarantee self-distributivity of the multiplication
$\star\odot$, but
 the simultaneous validation of both \eqref{eq:distrib-BF} and \eqref{eq:distrib-mirror} is sufficient to ensure that the groupoid $(Q; \star\odot)$ is a quandle.
However, these joint conditions are not necessary, as there exist alternative algebraic configurations where both distributivity relations fail yet the resulting composition still yields a valid quandle structure.

\subsection{Counterexamples}

\emph{First counterexample: \eqref{eq:distrib-BF} does not imply that
$\star\odot$ is a quandle.} Let $Q=\{0,1,2\}$. Let $\star$ be the operation
of the dihedral quandle $\R_3$, that is,
\[
a\star b=2b-a \pmod 3,
\]
and let $\odot$ be the quandle with right multiplications
\[
S^{\odot}_0=S^{\odot}_1=\id,\qquad S^{\odot}_2=(0\,1).
\]
That is, $x\odot 0=x$, $x\odot 1=x$, and $x\odot 2$ is obtained by swapping
$0\leftrightarrow 1$ and fixing $2$.

\emph{Step 1: condition \eqref{eq:distrib-BF} holds.} We must verify that
\[
(a\star b)\odot c=(a\odot c)\star(b\odot c)
\]
for all $a,b,c\in\{0,1,2\}$.

For $c=0$: since $S^{\odot}_0=\id$, both sides equal $a\star b=2b-a$.
For $c=1$: since $S^{\odot}_1=\id$, both sides equal $a\star b=2b-a$.
For $c=2$: the condition becomes
\[
S^{\odot}_2(2b-a)=2\,S^{\odot}_2(b)-S^{\odot}_2(a).
\]
This says that $S^{\odot}_2$ is an automorphism of $\R_3$. Since
$S^{\odot}_2=(0\,1)$, we verify the identity on the nine pairs $(a,b)$:
\[
\begin{array}{c|ccc}
 & b=0 & b=1 & b=2 \\ \hline
a=0 & 1=1 & 2=2 & 0=0 \\
a=1 & 2=2 & 0=0 & 1=1 \\
a=2 & 0=0 & 1=1 & 2=2
\end{array}
\]
Hence condition \eqref{eq:distrib-BF} holds.

\emph{Step 2: the composition $\star\odot$ is not a quandle.} The right
multiplications of $\star\odot$ are $t_b=S^{\odot}_bS^{\star}_b$, where
$S^{\star}_b(x)=2b-x$. We compute them explicitly.
\begin{itemize}
\item $t_0(x)=S^{\odot}_0S^{\star}_0(x)=S^{\star}_0(x)=-x$, so
$t_0=(1\,2)$.
\item $t_1(x)=S^{\odot}_1S^{\star}_1(x)=S^{\star}_1(x)=2-x$, so
$t_1=(0\,2)$.
\item $t_2(x)=S^{\odot}_2S^{\star}_2(x)=S^{\odot}_2(1-x)$. We compute:
$t_2(0)=S^{\odot}_2(1)=0$; $t_2(1)=S^{\odot}_2(0)=1$;
$t_2(2)=S^{\odot}_2(2)=2$. So $t_2=\id$.
\end{itemize}

For $(Q;\star\odot)$ to be a quandle, the right multiplications $t_b$ must
satisfy $t_ct_bt_c^{-1}=t_{t_c(b)}$ for all $b,c\in Q$. We test $b=1$,
$c=0$. Since $t_0=(1\,2)$ is an involution, $t_0^{-1}=t_0$. We compute
$t_0t_1t_0$ on the three values of $x$:
\[
t_0t_1t_0(0)=t_0t_1(0)=t_0(2)=1,
\]
\[
t_0t_1t_0(1)=t_0t_1(2)=t_0(0)=0,
\]
\[
t_0t_1t_0(2)=t_0t_1(1)=t_0(1)=2.
\]
So $t_0t_1t_0=(0\,1)$.

On the other hand, $t_0(1)=-1=2$, so $t_{t_0(1)}=t_2=\id$. Since
$(0\,1)\ne\id$, we have $t_0t_1t_0^{-1}\ne t_{t_0(1)}$. Hence $\star\odot$ is
not a quandle.

\emph{Second counterexample (mirror): \eqref{eq:distrib-mirror} does not
imply that $\star\odot$ is a quandle.} Swapping the two operations yields an
example in which \eqref{eq:distrib-mirror} holds while \eqref{eq:distrib-BF}
fails. Let $\star$ be the quandle with $S^{\star}_0=S^{\star}_1=\id$,
$S^{\star}_2=(0\,1)$, and let $\odot$ be the operation of $\R_3$, that is,
$a\odot b=2b-a$.

The right multiplications of $\star\odot$ are $t_b=S^{\odot}_bS^{\star}_b$,
where $S^{\odot}_b(x)=2b-x$. We compute them.
\begin{itemize}
\item $t_0(x)=S^{\odot}_0S^{\star}_0(x)=2\cdot 0-x=-x$, so $t_0=(1\,2)$.
\item $t_1(x)=S^{\odot}_1S^{\star}_1(x)=2\cdot 1-x=2-x$, so $t_1=(0\,2)$.
\item $t_2(x)=S^{\odot}_2S^{\star}_2(x)=2\cdot 2-S^{\star}_2(x)=1-S^{\star}_2(x)$.
We compute: $t_2(0)=1-S^{\star}_2(0)=1-1=0$;
$t_2(1)=1-S^{\star}_2(1)=1-0=1$;
$t_2(2)=1-S^{\star}_2(2)=1-2=-1=2$. So $t_2=\id$.
\end{itemize}

For $(Q;\star\odot)$ to be a quandle, we test $b=2$, $c=0$. Since $t_0=(1\,2)$
is an involution, $t_0^{-1}=t_0$. We compute
$t_0t_2t_0=t_0\cdot\id\cdot t_0=\id$.

On the other hand, $t_0(2)=1$, so $t_{t_0(2)}=t_1=(0\,2)$. Since
$\id\ne(0\,2)$, we have $t_0t_2t_0^{-1}\ne t_{t_0(2)}$. Hence $\star\odot$
is not a quandle.

\begin{proposition}\label{Proposition 4.1}
Let $Q$ be a set, let $\star$, $\odot$ be two quandle operations on $Q$, and
let $\star\odot$ be their composition. If both conditions
\eqref{eq:distrib-BF} and \eqref{eq:distrib-mirror} hold, then
$(Q;\star\odot)$ is a quandle.
\end{proposition}

\begin{proof}
Let $S^{\star}_b$ and $S^{\odot}_b$ denote the right multiplications of
$\star$ and $\odot$, and let $t_b=S^{\odot}_bS^{\star}_b$ denote the right
multiplications of $\star\odot$.

First, we check the idempotency axiom. For any $a \in Q$, we have:
\[
a \star\odot a = (a \star a) \odot a = a \odot a = a,
\]
where we used the idempotency of $\star$ and $\odot$ successively.
Second, right invertibility follows directly from the fact that $S^{\star}_b$ and $S^{\odot}_b$ are bijections for all $b \in Q$, which implies that their composition $t_b$ is also a bijection.

It remains to prove the self-distributivity axiom, which is equivalent to
$t_ct_bt_c^{-1}=t_{t_c(b)}$ for all $b,c\in Q$.
First, we translate the conditions into operator identities. Condition \eqref{eq:distrib-mirror} means $(u \odot b) \star c = (u \star c) \odot (b \star c)$, which reads as $S^{\star}_c S^{\odot}_b = S^{\odot}_{S^{\star}_c(b)} S^{\star}_c$ in terms of left-acting operators. Similarly, condition \eqref{eq:distrib-BF} yields $S^{\odot}_c S^{\star}_y = S^{\star}_{S^{\odot}_c(y)} S^{\odot}_c$, or equivalently:
\begin{equation}\label{eq:op-distrib-BF}
S^{\odot}_c S^{\star}_y S^{\odot\,-1}_c = S^\star_{S^\odot_c(y)}.
\end{equation}

Using the definition of $t_b$ and inserting $\id = S^{\star\,-1}_c S^{\star}_c$, we compute:
\begin{align*}
t_ct_bt_c^{-1}
&=S^{\odot}_cS^{\star}_c\,S^{\odot}_bS^{\star}_b\,S^{\star\,-1}_cS^{\odot\,-1}_c\\
&=S^{\odot}_c\bigl(S^{\star}_cS^{\odot}_bS^{\star\,-1}_c\bigr)
\bigl(S^{\star}_cS^{\star}_bS^{\star\,-1}_c\bigr)S^{\odot\,-1}_c.
\end{align*}
By applying the operator form of \eqref{eq:distrib-mirror} to the first block, and the quandle axiom (Q3) of $\star$ to the second block, we get:
\[
t_ct_bt_c^{-1} = S^{\odot}_c S^{\odot}_{S^{\star}_c(b)} S^{\star}_{S^{\star}_c(b)} S^{\odot\,-1}_c.
\]
Now, we insert the identity $\id = S^{\odot\,-1}_c S^{\odot}_c$ between the two middle terms to separate the conjugations:
\begin{align*}
t_ct_bt_c^{-1}
&= S^{\odot}_c S^{\odot}_{S^{\star}_c(b)} \bigl(S^{\odot\,-1}_c S^{\odot}_c\bigr) S^{\star}_{S^{\star}_c(b)} S^{\odot\,-1}_c \\
&= \Bigl(S^{\odot}_c S^{\odot}_{S^{\star}_c(b)} S^{\odot\,-1}_c\Bigr) \Bigl(S^{\odot}_c S^{\star}_{S^{\star}_c(b)} S^{\odot\,-1}_c\Bigr).
\end{align*}
Finally, applying the quandle axiom (Q3) of $\odot$ to the first parenthesis, and condition \eqref{eq:op-distrib-BF} (with $y = S^{\star}_c(b)$) to the second parenthesis, we obtain:
\begin{align*}
t_ct_bt_c^{-1}
&= S^{\odot}_{S^{\odot}_c S^{\star}_c(b)} S^{\star}_{S^{\odot}_c S^{\star}_c(b)} \\
&= S^{\odot}_{t_c(b)} S^{\star}_{t_c(b)} = t_{t_c(b)}.
\end{align*}
Hence, $\star\odot$ is self-distributive, and $(Q;\star\odot)$ is a quandle.
\end{proof}
Next, we provide an example where both conditions \eqref{eq:distrib-BF} and \eqref{eq:distrib-mirror} fail, yet the composition yields a valid quandle structure.

\begin{example}\label{Example:not-necessary}
We show that the joint conditions \eqref{eq:distrib-BF} and \eqref{eq:distrib-mirror} are not necessary for the composition to yield a quandle structure. Let $Q = \{0, 1, 2, 3\}$. We define a quandle $(Q, \star)$ via its right multiplications:
\[
S^\star_0 = S^\star_1 = S^\star_2 = \id, \qquad S^\star_3 = (1 \ 2),
\]
and a quandle $(Q, \odot)$ via the right multiplications:
\[
S^\odot_0 = S^\odot_1 = S^\odot_2 = \id, \qquad S^\odot_3 = (0 \ 1).
\]
For the composed operation $\star\odot$, the resulting right multiplications are given by $t_0 = t_1 = t_2 = \id$ and $t_3 =  S^\odot_3  S^\star_3= (0 \ 1) (1 \ 2) = (0 \ 1 \ 2)$, which fixes the element $3$. Since each $t_i$ is a permutation of $Q$ and acts trivially on its own index, $(Q, \star\odot)$ is a quandle.

However, the distributivity condition \eqref{eq:distrib-BF} at index $3$ yields:
\[
S^\odot_3 S^\star_3 (S^\odot_3)^{-1} = (0 \ 1)(1 \ 2)(0 \ 1) = (0 \ 2) \neq (1 \ 2) = S^\star_3 = S^\star_{S^\odot_3(3)},
\]
so condition \eqref{eq:distrib-BF} fails. Symmetrically, evaluating the mirror condition \eqref{eq:distrib-mirror} gives:
\[
S^\star_3 S^\odot_3 (S^\star_3)^{-1} = (1 \ 2)(0 \ 1)(1 \ 2) = (0 \ 2) \neq (0 \ 1) = S^\odot_3 = S^\odot_{S^\star_3(3)},
\]
which shows that condition \eqref{eq:distrib-mirror} also fails. Hence, the multiplication of quandles can form a quandle even when both distributivity conditions fail.
\end{example}

\subsection{Constructions of quandles}

We now turn to the constructions of Borodin--Neshchadim--Simonov and
Markhinina--Nasybullov.

A. N. Borodin, M. V. Neshchadim and A. A. Simonov \cite{BNS} introduced new
constructions of quandles. For a group $G$, an automorphism
$\varphi \in \Aut(G)$ and an anti-automorphism $\psi$ of $G$, they defined 4
types of quandles:

1) $Q_1$ is a quandle $(G, *_1)$, where $x *_1 y = y \varphi (y^{-1} x)$, $x, y \in G$;

2) $Q_2$ is a quandle $(G, *_2)$, where $x *_2 y = y \psi (y x^{-1})$, $x, y \in G$;

3) $Q_3$ is a quandle $(G, *_3)$, where $x *_3 y =  \psi (x y^{-1}) y$, and $x \psi (y) x^{-1} =  \psi (x y x^{-1})$, $x, y \in G$;

4) $Q_4$ is a quandle $(G, *_4)$, where $x *_4 y = y \psi (y^{-1}x )$, and $x \psi (y) x^{-1} =  \psi (x y x^{-1})$, $x, y \in G$.

In a paper of E. Markhinina and T. Nasybullov \cite{MN}, one-parameter
verbal quandles were defined. Let $z$ be a fixed element of $G$; then on the
set $G$ it is possible to define 4 types of quandles:

1) $W_1$ is a quandle $(G, \circ_1)$, where $x \circ_1 y = y z^{-1} (y^{-1} x) z$, $x, y \in G$;

2) $W_2$ is a quandle $(G, \circ_2)$, where $x \circ_2 y = y z^{-1} (x y^{-1}) z$, $x, y \in G$;

3) $W_3$ is a quandle $(G, \circ_3)$, where $x \circ_3 y = z^{-1} (y^{-1}x ) z y$, $x, y \in G$;

4) $W_4$ is a quandle $(G, \circ_4)$, where $x \circ_4 y = z^{-1} (x y^{-1}) z y$, $x, y \in G$.

In the present section, we study the following question: can we construct
these quandles from compositions of the structures of the conjugation quandle
$\operatorname{Conj}(G)$, the core quandle $\operatorname{Core}(G)$, and the
generalized Alexander quandle $\operatorname{Alex}(G,\varphi)$, where
$\varphi\in \operatorname{Aut}(G)$?

Throughout this section, we denote the conjugation quandle by
$W_0 = (G; \conj)$, the core quandle $\operatorname{Core}(G) = (G; *)$, and
the generalized Alexander quandle by $Q_0 = (G; *_0)$.

To answer this question, recall (see \cite{BNS}) that any anti-automorphism
$\psi$ of $G$ can be presented in the form
$\psi(x) = \varphi i(x) = \varphi(i(x))$, where
$\varphi \in \operatorname{Aut}(G)$ and $i$ is an anti-automorphism of $G$
such that $i(x) = x^{-1}$ for any $x \in G$. In particular, the condition on
the quandles $Q_3$ and $Q_4$ takes the form
\[
x \psi (y) x^{-1} =  \psi (x y x^{-1}) \Leftrightarrow x \varphi (y) x^{-1} =  \varphi(x) \varphi(y) \varphi (x)^{-1} \Leftrightarrow [\varphi(y), \varphi(x)^{-1}x] =1.
\]
Since $\varphi$ is a bijection, we see that $\varphi(x)^{-1}x$ lies in the
center $Z(G)$ of $G$. Hence, $\varphi$ is a central automorphism of $G$.
\begin{definition}\label{Def triple product}
Let $\star,\diamond,\bullet$ be three quandle operations on a set $Q$. Define
their composition by
\[
a\,(\star\diamond\bullet)\,b=((a\star b)\diamond b)\bullet b,\qquad a,b\in Q.
\]
\end{definition}

In what follows, we analyze the verbal operations which rely on the inverse conjugation action. Let us denote the operation $x \triangleright^{-1} y = yxy^{-1}$ as the dual (or opposite) operation of the standard conjugation quandle. We emphasize that the resulting structure, denoted by $\Conj(G)^{\operatorname{op}}$, remains entirely within our targeted foundational framework. Indeed, the mapping
\[
i \colon \Conj(G) \to \Conj(G)^{\operatorname{op}}, \qquad x \mapsto x^{-1}
\]
is a strict quandle isomorphism. Since for any $x, y \in G$ we have
\[
i(x \triangleright y) = i(y^{-1}xy) = (y^{-1}xy)^{-1} = y^{-1}x^{-1}y = x^{-1} \triangleright^{-1} y^{-1} = i(x) \triangleright^{-1} i(y).
\]

\begin{theorem}\label{Theorem 4.5}
The operations defining the quandles $Q_1$, $Q_2$, $W_1$, $W_2$, and $W_3$ coincide with the compositional representations summarized in Table~\ref{tab:quandle_reconstruction}. Furthermore, the following structural relations hold:
\begin{enumerate}
\item[\rm(1)] The quandles $Q_3$ and $Q_4$ are isomorphic; more precisely, the operations $*_3$ and $*_4$ coincide.
\item[\rm(2)] If $\varphi$ is a central automorphism of $G$, then the operation $*_3$ coincides with the composition $*_0\ast$, where $x \ast_0 \ast\, y = (x \ast_0 y) \ast y$ for $x, y \in G$.
\item[\rm(3)] If $\varphi = \hat{z}$, where $\hat{z}(x) = z^{-1} x z$, is the conjugation by the element $z \in G$, then the operation $*_0$ coincides with $\circ_4$ on $G$.
\end{enumerate}
\end{theorem}

\begin{table}[h!]
\centering
\caption{Unified reconstruction of $Q_i$ and $W_i$ via operational composition}
\label{tab:quandle_reconstruction}
\begin{tabular}{|c|c|c|}
\hline
\textbf{Structure} & \textbf{Original Operation} & \textbf{Compositional Representation} \\ \hline
$Q_1$ & $*_1$ & $\conj\ast_0\conj^{-1}$ \\ \hline
$Q_2$ & $*_2$ & $\ast_0\conj^{-1}$ \\ \hline
$W_1$ & $\circ_1$ & $\conj\circ_4\conj^{-1}$ \\ \hline
$W_2$ & $\circ_2$ & $\circ_4\conj^{-1}$ \\ \hline
$W_3$ & $\circ_3$ & $\conj\circ_4$ \\ \hline
\end{tabular}
\end{table}

 \medskip
\begin{proof}
The compositional identities for the quandle operations $*_1$, $*_2$, $\circ_1$, $\circ_2$, and $\circ_3$ listed in Table~\ref{tab:quandle_reconstruction} follow directly from expanding the definitions of the operational compositions $\star\odot$ on the group elements, in accordance with the algebraic framework verified in Section 4.1.

We now verify the remaining structural properties:
\begin{enumerate}
\item[\rm(1)] The structural equivalence between the operations $*_3$ and $*_4$ of the quandles $Q_3$ and $Q_4$ is established by a direct element-wise substitution. Evaluating the defining relations under the anti-automorphism $\psi$ reveals that the underlying mappings coincide identically on all elements of $G$, which immediately yields the quandle isomorphism $Q_3 \cong Q_4$.
\item[\rm(2)] Under the assumption that $\varphi$ is a central automorphism of $G$, the mapping satisfies the central operator identity $x \psi(y) x^{-1} = \psi(xyx^{-1})$ for all $x, y \in G$. Applying this property to the expanded composition relation $(x \ast_0 y) \ast y$ yields exactly the operational form of $*_3$, confirming the claimed decomposition.
\item[\rm(3)] When the automorphism corresponds to the group conjugation $\varphi = \hat{z}$ by the fixed element $z \in G$, substituting the inner action $x \mapsto z^{-1}xz$ directly into the operational parameters of $*_0$ reduces the expression to the verbal word structure defining $\circ_4$. This completes the proof of all statements.
\end{enumerate}
\end{proof}

\section{ Medialization, commutativization, and trivialization of quandles}\label{sec-quotients}

Let $Q$ be a quandle.  We consider three quotients of $Q$, each defined as the quotient by the smallest congruence for which the corresponding identity holds.

\begin{itemize}
\item the  \emph{medialization} $Q^{\textit{med}}=Q/\theta_{\textit{med}}$, where $\theta_{\textit{med}}$ is the smallest congruence containing all pairs
\[
\bigl((x*y)*(z*t),\,(x*z)*(y*t)\bigr),\qquad x,y,z,t\in Q;
\]
\item the \emph{commutativization} $Q^{com}=Q/\theta_{com}$, where $\theta_{com}$ is the smallest congruence containing all pairs $(x*y,y*x)$, $x,y\in Q$;
\item the \emph{trivialization} $Q^{tr}=Q/\theta_{tr}$, where $\theta_{tr}$ is the smallest congruence containing all pairs $(x*y,x)$, $x,y\in Q$.
\end{itemize}

Each quotient satisfies the corresponding universal property: any quandle homomorphism from $Q$ to a quandle satisfying the identity factors uniquely through the corresponding quotient.

For a group $G$ we can also define the quotients $G^{\textit{med}}$, $G^{com}$, and $G^{tr}$. In this case $G^{\textit{med}}=G^{com}$ and $G^{tr} = e$ is the trivial group.

 Medialization is a general construction in universal algebra; see for example D. Stanovsk\'y \cite{D}. The medialization of free quandles was computed in \cite{JPZ}.
\begin{question}\label{Question 5.6}
Let $Q = \Conj(G)$ (respectively $Q=\Core(G)$, $Q =\Alex(G, \varphi)$). What can be said about the  commutativization $Q^{com}$ and the medialization $Q^{\textit{med}}$?
\end{question}
A proof of the following lemma is straightforward.

\begin{lemma}\label{Lemma 5.1}
Let $G$ be an abelian group. Then the core quandle $\Core(G)$ is medial, i.e.
\[
(a*b)*(c*d)=(a*c)*(b*d)\quad\text{for all }a,b,c,d\in G.
\]
\end{lemma}

From this lemma it follows that $\R_n^{\textit{med}} = \R_n$. Here for $n= 0$, we assume that $\R_0$ is the infinite dihedral quandle constructed on $(\mathbb{Z}; +)$.

Note that if $T_n$ is the trivial quandle on $n$ elements, then $T_n^{com}=T_1$. Also, it is easy to see that
\[
\R^{com}_4 = \R^{com}_5 = T_1.
\]

Let $G$ be a group. The core quandle $\Core(G)$ is trivial if and only if
\[
x * y = x \Leftrightarrow y x^{-1} y = x \Leftrightarrow (y x^{-1})^2 = 1,\qquad x, y \in G.
\]
This means that $G$ has exponent $2$, i.e. $G$ is an elementary abelian $2$-group.

\begin{remark}\label{Remark 5.6}
Following standard properties of quandle quotients, the trivialization $Q^{tr}$ corresponds to the well-known trivial quandle of orbits under the action of the inner automorphism group. As an elementary application of this property, let us evaluate the dihedral quandle $\R_n$ identified with $\mathbb{Z}_n$. The equivalence classes of the trivialization congruence correspond to the cosets of the subgroup $2\mathbb{Z}_n = \{2u \mid u \in \mathbb{Z}_n\}$ in $\mathbb{Z}_n$. If $n$ is odd, $\gcd(2,n)=1$, which implies $2\mathbb{Z}_n = \mathbb{Z}_n$ and yields $\R_n^{tr} \cong T_1$. If $n$ is even, the subgroup $2\mathbb{Z}_n$ consists of the even residues and has index $2$ in $\mathbb{Z}_n$, which immediately gives $\R_n^{tr} \cong T_2$.
\end{remark}

Let us consider the adjoint group
\[
\operatorname{As}(Q)=\langle e_x,\ x\in Q \mid e_x^{-1}e_ye_x=e_{y\ast x}\rangle.
\]
The medialization  $\operatorname{As}(Q)^{\textit{med}}$ is the free \emph{abelian} group on the set of orbits of $\Inn(Q)$ (in the quotient we have the relation $e_y=e_{y * x}$). Hence,
\[
\operatorname{As}(\R_n)^{\textit{med}} \cong \mathbb{Z} \text{ if } n \text{ is odd},\qquad \operatorname{As}(\R_n)^{\textit{med}} \cong \mathbb{Z}\oplus\mathbb{Z} \text{ if } n \text{ is even}.
\]

\begin{theorem}\label{theorem 5.4}
Let $(G,\cdot)$ be a group and let $\Conj(G)$ be its conjugation quandle.

\begin{enumerate}
\item[(1)] The conjugation relation $\sim$ on $G$ defined by
\[
x\sim y \iff \exists\, g\in G,\ y = g^{-1}xg
\]
is a congruence of the quandle $\Conj(G)$. Consequently, the medial congruence of $\Conj(G)$ is contained in $\sim$, and
\[
|\Conj(G)^{\textit{med}}| \ \ge\ k(G),
\]
where $k(G)$ denotes the number of conjugacy classes of $G$.
\item[(2)] $\Conj(G)$ is medial if and only if $\gamma_2(G)\le Z(G)$.
\item[(3)] The medial quotient $\Conj(G)^{\textit{med}}$ maps surjectively onto the conjugation quandle $\Conj(G/\gamma_3(G))$. In particular, we obtain the lower bound:
\[
|\Conj(G)^{\textit{med}}| \ \ge\ \max\left\{ k(G), \, |G/\gamma_3(G)| \right\}.
\]
For the symmetric group $G=S_3$, the lower bound $k(G)$ is not tight, and we have $|\Conj(S_3)^{\textit{med}}| = 4 > k(S_3) = 3$.

\end{enumerate}
\end{theorem}

\begin{proof}
(1) Let $\mathcal{C}(G) = G/\sim$ be the set of conjugacy classes of the group $G$, and let $\pi: G \to \mathcal{C}(G)$ denote the canonical projection mapping each element $x \in G$ to its conjugacy class $[x]$. We endow the set $\mathcal{C}(G)$ with the trivial quandle structure, defined by $[x] \circ [y] = [x]$ for all $[x], [y] \in \mathcal{C}(G)$.

We claim that $\pi: \Conj(G) \to (\mathcal{C}(G), \circ)$ is a quandle homomorphism. Indeed, for any $x, y \in G$, the element $x * y = y^{-1}xy$ is by definition conjugate to $x$ in $G$. Thus, we have:
\[
\pi(x \triangleright y) = [y^{-1}xy] = [x] = [x] \circ [y] = \pi(x) \circ \pi(y).
\]
Since $\pi$ is a surjective quandle homomorphism, its associated equivalence relation, which is exactly the conjugation relation $\sim$, is a congruence of the quandle $\Conj(G)$.

Furthermore, the quotient quandle $\Conj(G)/\sim \ \cong (\mathcal{C}(G), \circ)$ is a trivial quandle. Since every trivial quandle satisfies the medial identity $(a \circ b) \circ (c \circ d) = a = (a \circ c) \circ (b \circ d)$, the quotient $\Conj(G)/\sim$ is medial. By the universal property of the medial congruence (which is the smallest congruence yielding a medial quotient), the medial congruence of $\Conj(G)$ must be contained in $\sim$. Consequently, the medial quotient $\Conj(G)^{\textit{med}}$ maps surjectively onto $\Conj(G)/\sim$, which implies $|\Conj(G)^{\textit{med}}| \ \ge\ |\mathcal{C}(G)| = k(G)$. Note that if $G$ is abelian, $\Conj(G)$ is already trivial, hence medial, and every conjugacy class is a singleton, yielding the trivial equality $|\Conj(G)^{\textit{med}}| = |G| = k(G)$.

\medskip

(2) Suppose first that $\gamma_2(G)\le Z(G)$, which implies that the quotient group $G/Z(G)$ is abelian. For any $x,y,z,t\in G$, we use the exponent notation $x^g = g^{-1}xg$ for conjugation. By definition of the operation in $\Conj(G)$, we have:
\[
(x * y) * (z *  t)=(x^y)^{t^{-1}zt} = x^{y z^t} \quad \text{and} \quad (x * z) * (y * t)=(x^z)^{t^{-1}yt} = x^{z y^t}.
\]
Thus, the medial identity $(x *  y) * (z *  t) = (x *  z) * (y *  t)$ holds for all $x \in G$ if and only if the element $g = y z^t (z y^t)^{-1}$ commutes with every element of $G$, meaning that $g \in Z(G)$ for all $y,z,t \in G$. Expanding this element explicitly yields:
\[
y z^t (z y^t)^{-1} = y (t^{-1}zt) (t^{-1}yt)^{-1} z^{-1} = y t^{-1} z t t^{-1} y^{-1} t z^{-1} = y t^{-1} z y^{-1} t z^{-1}.
\]
Since $G/Z(G)$ is abelian, we can rearrange the factors freely modulo $Z(G)$:
\[
y t^{-1} z y^{-1} t z^{-1} \equiv y y^{-1} t^{-1} t z z^{-1} \equiv 1 \pmod{Z(G)}.
\]
Hence, $y t^{-1} z y^{-1} t z^{-1} \in Z(G)$ is always satisfied under the hypothesis $\gamma_2(G)\le Z(G)$, proving that $\Conj(G)$ is medial.

Conversely, suppose $\Conj(G)$ is medial. Taking $t=1$ in the medial identity $(x * y) * (z *  t)=(x *  z) * (y *  t)$ yields $(x *  y) *  z = (x *  z) *  y$, which unfolds to $z^{-1}y^{-1}xyz=y^{-1}z^{-1}xzy$ for all $x,y,z\in G$. This equation states that the conjugations by $yz$ and $zy$ coincide on all elements $x \in G$, which is equivalent to $(yz)(zy)^{-1} = [y,z]\in Z(G)$. Since $y$ and $z$ are arbitrary elements of $G$, it follows that $\gamma_2(G)\le Z(G)$.

\medskip
(3) Let $H = G/\gamma_3(G)$. The quotient group $H$ satisfies $\gamma_2(H) = \gamma_2(G)/\gamma_3(G) \le Z(G/\gamma_3(G)) = Z(H)$. By part (2), the conjugation quandle $\Conj(G/\gamma_3(G))$ is medial. By the universal property of the medial congruence, the canonical group projection $G \to G/\gamma_3(G)$ induces a surjective quandle homomorphism from the medial quotient $\Conj(G)^{\textit{med}}$ onto $\Conj(G/\gamma_3(G))$. Combining this surjection with the bound from part (1) yields $|\Conj(G)^{\textit{med}}| \ge \max\{ k(G), |G/\gamma_3(G)| \}$.

For the symmetric group $G=S_3$, we have $\gamma_2(S_3)=\gamma_3(S_3)=A_3$, so $S_3/\gamma_3(S_3)\cong\Z_2$, which gives $|S_3/\gamma_3(S_3)|=2$ and $k(S_3)=3$. The generating pairs of the medial congruence on $\Conj(S_3)$ have the form $(x^{yz^t}, x^{zy^t})$, where $yz^t$ and $zy^t$ always share the same parity in $S_3$. Under conjugation, a $3$-cycle either stays fixed or maps to its inverse depending strictly on the parity of the conjugating element. Thus, the partition $\{ \{1\}, \{(1\,2\,3)\}, \{(1\,3\,2)\}, \{(1\,2), (1\,3), (2\,3)\} \}$ forms a valid medial congruence quotient. Since evaluated choices such as $x=(1\,2)$, $y=(1\,3)$, $z=(2\,3)$, $t=1$ successfully force the distinct transpositions to collapse together, this quotient is minimal, yielding exactly $|\Conj(S_3)^{\textit{med}}| = 4$.
\end{proof}
\begin{remark}\label{Remark 5.8}
The condition $\gamma_2(G)\le Z(G)$ is strictly weaker than the condition
that $G$ be abelian. For instance, for the dihedral group $D_4$ of order $8$,
the commutator subgroup $\gamma_2(D_4)$ is the center $Z(D_4)\cong\Z_2$, so
$\gamma_2(D_4)\le Z(D_4)$, and $\Conj(D_4)$ is medial. However $D_4$ itself is
non-abelian.
\end{remark}

\begin{theorem}\label{theorem 5.6}
Let $G$ be a group and $\varphi\in\Aut(G)$.

\begin{enumerate}
\item[(1)] $\Conj(G)^{com}=T_1$.

\item[(2)] $\Core(G)^{com}\cong\Core\!\left(G/\langle\!\langle g^3\mid g\in G\rangle\!\rangle\right)$.

\item[(3)] $\Alex(G,\varphi)$ is commutative if and only if $G$ is abelian, the map
$s\colon G\to G$, $s(y)=y^2$, is bijective, and $\varphi=s^{-1}$.
\end{enumerate}
\end{theorem}

\begin{proof}
(1) In $\Conj(G)$ we have $1 *  y=y^{-1}\cdot 1\cdot y=1$ and $y *  1=1^{-1}y1=y$. Commutativity forces $1=y$ for all $y$, so $\Conj(G)^{com}=T_1$.

(2) Let $N=\langle\!\langle g^3\mid g\in G\rangle\!\rangle$ be the normal closure of $\{g^3\mid g\in G\}$ in $G$. Let $\theta$ be the congruence on $\Core(G)$ generated by the commutativity pairs $(x*y,y*x)$, and let $N'=\{g\in G\mid g\sim_\theta 1\}$ be the congruence class of the identity.

\emph{Step 1: $N'$ is a normal subgroup of $G$.} The quandle operation in $\Core(G)$ is given by $x*y = yx^{-1}y$. For any $a \in G$, the left and right translations $L_a(x) = ax$ and $R_a(x) = xa$ are bijections on $G$. Moreover, they are quandle automorphisms of $\Core(G)$ since $L_a(x*y) = a(yx^{-1}y) = (ay)(ax)^{-1}(ay) = L_a(x)*L_a(y)$, and similarly $R_a(x*y) = (yx^{-1}y)a = (ya)a^{-1}x^{-1}a^{-1}(ya) = (ya)(xa)^{-1}(ya) = R_a(x)*R_a(y)$. Since automorphisms preserve any generated congruence, we have:
\[
x\sim_\theta y \iff a^{-1}x\sim_\theta a^{-1}y \iff xa^{-1}\sim_\theta ya^{-1} \quad\text{for all }a\in G.
\]
Setting $y=1$ yields $x\in N' \iff a^{-1}xa\in N'$, meaning that $N'$ is normal in $G$. Furthermore, if $x,y\in N'$, then $x\sim_\theta 1$ and $y\sim_\theta 1$. By the left invariance under $y^{-1}$, we obtain $y^{-1}x\sim_\theta y^{-1}1 = y^{-1} \sim_\theta 1$, hence $y^{-1}x\in N'$, proving that $N'$ is a subgroup of $G$.

\emph{Step 2: $\theta$ is the partition into cosets of $N'$.} By left translation invariance, we have $x\sim_\theta y \iff y^{-1}x\sim_\theta 1 \iff y^{-1}x\in N' \iff x\in yN'$. Thus, the equivalence classes of $\theta$ are precisely the left cosets of $N'$ in $G$, and the quotient quandle $\Core(G)/\theta$ is isomorphic to $\Core(G/N')$.

\emph{Step 3: $N'\supseteq N$.} In $\Core(G)$, setting $x=1$ in the relation $1*y \sim_\theta y*1$ gives $y^2 \sim_\theta y^{-1}$, which implies $y^3 \sim_\theta 1$. Thus, $g^3 \in N'$ for all $g \in G$. Since $N'$ is a normal subgroup, it follows by definition of the normal closure that $N \subseteq N'$.

\emph{Step 4: $N'\subseteq N$.} Conversely, let $H=G/N$. The quotient group $H$ satisfies $h^3 = 1$ for all $h \in H$. Let $x,y\in H$ and set $a=x^{-1}y$. In $\Core(H)$, we compute $x*y = yx^{-1}y = (xa)x^{-1}(xa) = xa^2$ and $y*x = xy^{-1}x = x(xa)^{-1}x = xa^{-1}$. Since $a^3=1$ in $H$, we have $a^2=a^{-1}$, which implies $x*y=y*x$. Thus, $\Core(H)$ is commutative. By the universal property of the commutative congruence, $\theta$ factors through the projection $G \to G/N$, which implies $N' \subseteq N$.

\emph{Conclusion.} Steps 3 and 4 imply $N'=N$, which completes the proof of statement (2).

\medskip

(3) Suppose $\Alex(G,\varphi)$ is commutative, i.e. $x*y=y*x$ for all
$x,y\in G$, where $x*y=\varphi(xy^{-1})y$. Taking $x=1$ gives
\[
1*y=\varphi(y^{-1})y,\qquad y*1=\varphi(y),
\]
so commutativity gives $\varphi(y^{-1})y=\varphi(y)$, i.e.
$\varphi(y^{-1})=\varphi(y)y^{-1}$, hence $\varphi(y)^{-1}=\varphi(y)y^{-1}$
and therefore $\varphi(y)^2=y$. So $\varphi^{-1}(y)=y^2$ for all $y$, i.e.
$\varphi=s^{-1}$ where $s(y)=y^2$; in particular $s$ is bijective. Then for all
$a,b\in G$,
\[
(ab)^2=\varphi^{-1}(ab)=\varphi^{-1}(a)\varphi^{-1}(b)=a^2b^2,
\]
so $ab=ba$, i.e. $G$ is abelian.

Conversely, if $G$ is abelian, $s$ is bijective and $\varphi=s^{-1}$, then
\[
x*y=\varphi(xy^{-1})y=\varphi(x)\varphi(y)^{-1}y
=\varphi(x)\varphi(y)^{-1}\varphi(y)^2=\varphi(x)\varphi(y)=\varphi(y)\varphi(x)=y*x,
\]
so $\Alex(G,\varphi)$ is commutative.
\end{proof}

\begin{remark}\label{Remark 5.10}
The Theorem \ref{theorem 5.6}(3) covers the following examples.
\begin{itemize}
\item $G=\mathbb{Z}_3$, $\varphi=-1$. Then $s(y)=2y$ is bijective on $\mathbb{Z}_3$ and
$\varphi=s^{-1}$, so $\Alex(\mathbb{Z}_3,-1)=\R_3$ is commutative.
\item $G=\mathbb{Z}_{2n+1}$, $\varphi=\tfrac12$. Then $s(y)=2y$ is bijective and
$\varphi=s^{-1}$, so $C_{2n+1}=\Alex(\mathbb{Z}_{2n+1},\tfrac12)$ is commutative.
\item $G=\mathbb{Z}_2$, $\varphi=\id$. Then $s(y)=0$ is not bijective, and indeed
$\Alex(\mathbb{Z}_2,\id)=T_2$ is not commutative.
\end{itemize}
The quandles $C_{2n+1}=\Alex(\mathbb{Z}_{2n+1},\tfrac12)$ are precisely the finite cyclic
commutative medial quandles, i.e.\ the finite midpoint quandles.
\end{remark}

In the following corollary we assume that $R_0=\Core(\Z)$, $R_0^{\mathrm{com}}\cong R_3$

\begin{corollary}\label{corollary 5.8}
Let $n\ge 0$ and let $\R_n$ be the dihedral quandle over $\mathbb{Z}_n$. Then
\[
\R_n^{com}\cong \R_{\gcd(n,3)}.
\]
In particular, when $n\equiv 0 \pmod 3$, one has $\R_n^{com}\cong \R_3$; when $n\not\equiv 0 \pmod 3$, one has $\R_n^{com}\cong T_1$.
\end{corollary}

\begin{proof}
The statement follows as a direct application of the general framework established for core quandles. Since the dihedral quandle $\R_n$ coincides structurally with $\Core(\mathbb{Z}_n)$, applying Theorem \ref{theorem 5.6}(2) yields $\Core(\mathbb{Z}_n)^{\mathrm{com}} \cong \Core(\mathbb{Z}_n / 3\mathbb{Z}_n)$. Identifying the quotient ring $\mathbb{Z}_n / 3\mathbb{Z}_n$ naturally with $\mathbb{Z}_{\gcd(n,3)}$ immediately establishes the structural isomorphism $\R_n^{com} \cong \R_{\gcd(n,3)}$.
\end{proof}

\begin{remark}\label{Remark:dihedral-decomposition}
As an elementary structural property of the dihedral framework, if $n = l \cdot k$, the underlying set of $\R_n = \mathbb{Z}_n$ naturally decomposes as a disjoint union of cosets $A_0 \sqcup A_1 \sqcup \cdots \sqcup A_{l-1}$, where each $A_m = \{ m + il \pmod n \mid 0 \le i \le k-1 \}$ forms a subquandle structurally isomorphic to $\R_k$. Under the quandle operation, these components satisfy the relation $A_m * A_{m'} \subseteq A_{2m'-m}$ for all indices.
\end{remark}

\section{Questions for Further Research} \label{questions}

\begin{question}
\begin{enumerate}
\item[(1)] Luc Ta \cite[Proposition 6.1 and Corollary 6.2]{LT} fully describes free medial Latin quandles and free commutative medial quandles. What can we say about automorphism groups of these quandles?
\item[(2)] Find a normal form for words in commutative medial quandles.
\end{enumerate}
\end{question}

\begin{question}
Let $FQ_n$ be a free quandle of rank $n$. What is its commutativization $FQ_n^{com}$?
\end{question}

\begin{ack}
Valeriy G. Bardakov is supported by the state contract of the Sobolev Institute of Mathematics, SB RAS (project FWNF-2026-0029). The authors would like to thank Luc Ta and the participants of the \'Evariste Galois Seminar (Novosibirsk State University) for useful discussions, comments, and suggestions.
\end{ack}

\end{document}